\documentclass[12pt,reqno]{amsart}
\usepackage{amsmath, amssymb, amsthm}
\usepackage{url}
\usepackage[breaklinks]{hyperref}
\usepackage{hyperref}
\usepackage{amssymb}

\renewcommand{\Re}{\operatorname{Re}}

\renewcommand{\(}{\left\(}
\renewcommand{\)}{\right\)}
\renewcommand{\[}{\left\[}
\renewcommand{\]}{\right\]}

\numberwithin{equation}{section}
\theoremstyle{plain}
\newtheorem{theorem}{Theorem}[section]
\newtheorem{lemma}[theorem]{Lemma}
\newtheorem{remark}[]{Remark}

\usepackage{amsmath}
\usepackage{mathtools}

\begin{document}
	\title[Laurent series expansions of $L$-functions]{Laurent series expansions of $L$-functions}

	\author{Tushar Karmakar}
	\address{Tushar Karmakar\\ Department of Mathematics \\
Tulane University,  New Orleans, LA 70118,  USA. } 
\email{tkarmakar@tulane.edu}

	 \author{Saikat Maity}
\address{Saikat Maity\\ Department of Mathematics \\
Indian Institute of Technology Indore \\
Indore, Simrol, Madhya Pradesh 453552, India.} 
\email{saikat.maths1729@gmail.com}
	
	 \author{Bibekananda Maji}
\address{Bibekananda Maji\\ Department of Mathematics \\
Indian Institute of Technology Indore \\
Indore, Simrol, Madhya Pradesh 453552, India.} 
\email{bibek10iitb@gmail.com, bmaji@iiti.ac.in}


	\maketitle
	
%
	\begin{abstract}
	One of the main objectives of the current paper is to revisit the well known Laurent series expansions of the Riemann zeta function $\zeta(s)$,  Hurwitz zeta function $\zeta(s,a)$ and Dirichlet $L$-function $L(s,\chi)$ at $s=1$.  Moreover,  we also present a new Laurent series expansion of $L$-functions associated to cusp forms over the full modular group.  
	\end{abstract}
	
	\section{Introduction}

Finding a Laurent series expansion of a meromorphic function has always been a question of curiosity for mathematicians.  Riemann \cite{Riemann}, in his seminal paper,  showed that the Riemann zeta function $\zeta(s)$ can be analytically continued to the whole complex plane except at $s=1$.  Moreover,  at $s=1$,  $\zeta(s)$ has a simple pole with residue $1$.  
	In 1955,  Briggs and Chowla \cite{Briggs and Chowla} gave two different proofs for the Laurent series expansion of $\zeta(s)$ at $s=1$.  Mainly,  they showed that
	\begin{align*}
	\zeta(s)= \frac{1}{s-1}+ \sum_{n=0}^\infty \frac{(-1)^n \gamma_n}{n!} (s-1)^n,
	\end{align*}
	where $\gamma_0=\gamma$ is the well-known Euler–Mascheroni constant,  and for $n \geq 1$,  $\gamma_n$ is defined by 
	\begin{align*}
	\gamma_n = \lim_{N \rightarrow \infty} \left\{ \sum_{\ell=1}^N \frac{\log^n(\ell)}{\ell} - \frac{\log^{n+1}(N)}{n+1} \right\}.
	\end{align*}
	The Dirichlet $L$-function $L(s,  \chi)$ and Hurwitz zeta function $\zeta(s,a)$ are two well-known generalizations of $\zeta(s)$.  If $\chi$ is not a principal character,  $L(s,\chi)$ can be analytically continued to the whole $\mathbb{C}$.  
	Like $\zeta(s)$,  $\zeta(s,a)$ can be  analytically continued to
	 the whole complex plane except for a simple pole at $s = 1$.   In 1972,   Berndt \cite{Berndt} derived the Laurent series expansion of $\zeta(s,a)$ at  $s=1$.  In 1961,  Briggs and Buschman \cite{Briggs-Buschman} studied Laurent series of a general Dirichlet series having a simple pole.  Readers are also encouraged to see \cite{Briggs62,  Knopfmacher,  Shirasaka}  and references therein.  In this paper,  we present simple methods to prove the Laurent series expansion of $\zeta(s)$,  $L(s, \chi)$ and $\zeta(s,a)$.  Furthermore,  we obtain a new Laurent series expansion of $L$-functions associated to cusp forms over the full modular group.  
	
	Let $k$ be an even positive integer and $\mathbb{H}$ denote the upper half plane. Let $S_k (SL_2 (\mathbb{Z}))$ denote the space of cusp form of weight $k$ over the full modular group $SL_2 (\mathbb{Z})$. Let $f (z) \in S_k (SL_2 (\mathbb{Z}))$ be a normalized Hecke eigenform with the Fourier series expansion,  
\begin{align}\label{Fouries series}	
	f (z) = \sum_{n=1}^{\infty} a_f (n) \exp (2\pi inz), \hspace{0.4cm} \forall z \in \mathbb{H}.
	\end{align}
	The $L$-function associated to the cusp form $f$ is defined as 
	$$
	L(f,s)= \sum_{n=1}^{\infty} \frac{a_f(n)}{n^s}. 
	$$
	This $L$-function is absolutely convergent for $\Re(s) > \frac{k+1}{2}$ as we know from
	Delinge’s bound $a_f (n) = O(n^{\frac{(k-1)}{2}} + \epsilon)$ for any $\epsilon > 0$.
	Hecke proved that $L( f ,s)$ can be analytically continued to an entire function and it
	satisfies following functional equation:
	\begin{align}\label{functional equn}
	(2\pi)^{-s} \Gamma(s) L(f,s) = i^k (2\pi)^{-(k-s)} \Gamma(k-s) L(f,k-s).
	\end{align}
This functional equation is equivalent to the following transformation formula,  for $\Re(y)>0$, 
\begin{align}\label{Analogue of Jacobi theta transformation}
y^k W(y)= W(1/y),
\end{align}
where 
\begin{align*}
W(y)= \sum_{n=1}^{\infty} a_f (n) \exp(-2\pi ny).  
\end{align*}
Readers can see \cite{ Bochner, KC} to know more about different identities equivalent to the functional equation \eqref{functional equn}.  
One of the main aims of this paper is to derive a Laurent series expansion of $L(f,s)$ at $s=0$.  
	
	 Now we define the upper incomplete gamma function $\Gamma(s,  a)$,  for $a>0$,  $s \in \mathbb{C}$,  as follows:
	\begin{align}\label{incomplete gamma}
	\Gamma(s,  a) = \int_{a}^\infty e^{-y} y^{s-1} dy.  
	\end{align}
As $ a \rightarrow \infty$,  we have the following asymptotic expansion\cite[p. ~179,  Equation (8.11.2)]{NIST}:
\begin{align}\label{asymp1}
\Gamma(s,  a) \sim a^{s-1} e^{-a}. 
\end{align}	
More generally,   for $a>0,  \ell \geq 0,  s \in \mathbb{C}$,  we define  the following one variable generalization of the above incomplete gamma function: 
\begin{align}\label{gen incomplete gamma}
\Gamma_\ell(s,  a):= \int_{a}^\infty e^{-y} \log^{\ell}(y) y^{s-1} dy.    
\end{align}
For $ a \rightarrow \infty$,  we have the following asymptotic expansion:
\begin{align}\label{asymp2}
\Gamma_{\ell}(s,  a) \sim a^{s-1} \log^{\ell}(a)  e^{-a}. 
\end{align}	
	Now we are ready to state main results of this paper. 
	\section{Main Results}
	First,  we will present simple proofs of the Laurent series expansions of $\zeta(s),  \zeta(s,  a)$ and $L(s, \chi)$.  	
	\begin{theorem} \label{Laurent series coefficients for Riemann zeta function}
		The Laurent series of the Riemann zeta function $\zeta(s)$ at its pole $s=1$ is given by 
		$$\zeta(s) = \frac{1}{s-1} + \sum_{k=0}^{\infty} \frac{(-1)^k}{k!} \gamma_k (s-1)^k,$$
		where 
\begin{align}\label{gamma_k}		
		\gamma_k = \lim_{m \to \infty}\left({\sum_{i=1}^{m}{\frac{\log^k(i)}{i}}}- \frac{\log^{k+1}(m)}{k+1}\right).
		\end{align}
	\end{theorem} 
	
	\begin{theorem} \label{Laurent series coefficients for Hurwitz zeta function}
		The Laurent series of the Hurwitz zeta function $\zeta(s,a)$ at its pole $s=1$ is given by 
		$$\zeta(s,  a) = \frac{1}{s-1} + \sum_{k=0}^{\infty} \frac{(-1)^k}{k!} \gamma_k (a) (s-1)^k,$$
		where 
		$$\gamma_{k}(a)= \lim_{m \to \infty} \left\{ \sum_{i=0}^{m} \frac{(\log(i+a))^{k}}{(i+a)} - \frac{(\log(m+a))^{k+1}}{k+1} \right\}.$$
	\end{theorem} 
	
	\begin{theorem} \label{Laurent series coefficients for Dirichlet l=function}
		Let $\chi$ be a non-principal Dirichlet character modulo $q$.  The  Laurent series of the Dirichlet $L$-function $L(s,\chi)$ at $s=1$ is given by 
		$$L(s,\chi) = \sum_{k=0}^{\infty} \frac{(-1)^k}{k!} \gamma_k (\chi) (s-1)^k,$$ where 
		$$\gamma_k (\chi) = \sum_{a=1}^{q} \chi(a) \gamma_k (a,q),$$
		where \begin{align}\label{gamma_k(a,q)}
\gamma_k(a,  q) = \lim_{x \rightarrow \infty} \sum_{\substack{n \leq x, \\ n \equiv a \pmod{q}} } \left( \frac{\log^k(n)}{n} - \frac{\log^{k+1}(x)}{q(k+1)}    \right).
\end{align}
		\end{theorem} 
	The next result gives a Laurent series expansion of $L(f,s)$ at $s=0$.  
	\begin{theorem} \label{Laurent series coefficients for $L$-Function associated to cusp form of weight $k$}
		Let $f(z)$ be a cusp form of weight $k$ over $SL_2(\mathbb{Z})$ with the Fourier series expansion \eqref{Fouries series}.    
		Then the L-function $L(f, s)$ has the following Laurent series expansion at $s=0$, 
		\begin{align*}
		L(f, s) = \sum_{n=1}^\infty C(n, k) s^n,
		\end{align*}
	where the first two coefficients are given  by
	\begin{align}
	C(1,k) & =  \sum_{n=1}^\infty \frac{ a_f(n) \Gamma(k,  2 n \pi) }{(2n\pi)^k} + \sum_{n=1}^\infty  a_f(n) \Gamma(0,  2 n \pi),  \label{Final C(1,k)} \\
	C(2,k)&=  C(1,k) \left(  \gamma + \log(2\pi) \right) -  \sum_{n=1}^\infty a_f(n) \left(  \frac{\Gamma_1(k,  2n \pi) -\log(2n \pi) \Gamma(k,  2n\pi)}{(2n\pi)^k}   \right) \nonumber \\
& +  \sum_{n=1}^\infty a_f(n) \left( \Gamma_1(0,  2n \pi) -\log(2n \pi) \Gamma(0,  2n\pi)  \right),  \label{Final C(2,k)}
\end{align}		
where incomplete gamma functions $\Gamma(k, 2n\pi)$ and $\Gamma_1(k,  2 n \pi)$  are defined as in \eqref{incomplete gamma} and \eqref{gen incomplete gamma},  respectively.  
		\end{theorem}
\begin{remark}		
		Note that the coefficient $C(1,k)$ is nothing but the first derivatives of the $L$-function $L(f,s)$ at $s=0$.  Therefore,  the above theorem gives us a new way to evaluate the derivative of  $L(f,s)$ at $s=0$.  The series present in \eqref{Final C(1,k)}-\eqref{Final C(2,k)} are rapidly convergent.  
		\end{remark}

		\section{Preliminaries} 
In this section,  we mention a few well known results that will be useful to obtain our results.	
		\begin{lemma}\cite[p.~56]{Apostol} \label{Integral representation of Riemann zeta function}
		For $\Re(s) >0$, $\zeta(s)$ satisfies the following integral representation,
		\begin{equation*}
			\zeta(s) = \frac{s}{s-1} - s \int_{1}^{\infty} \frac{x- [x]}{x^{s+1}} dx.  
		\end{equation*}
	\end{lemma}
	
	\begin{lemma}\cite[p.~269,  Theorem 12.21]{Apostol}\label{Integral representation of Hurwitz zeta function}
		For any non negative integer $N$ and $\Re(s) >0$, 
		\begin{equation*}
			\zeta(s,a)= \sum_{n=0}^{N}\frac{ 1}{(n+a)^s} + \frac{ (N+a)^{1-s}}{s-1} -s \int_{N}^{\infty}\frac{x-[x]}{(x+a)^{s+1}}  dx. 
		\end{equation*}
	\end{lemma}
Next,  we state an important result which will play crucial role in obtaining the Laurent series expansion of $L(s, \chi)$ at $s=1$.  	
\begin{lemma}\cite[Proposition 3.2]{Knopfmacher}\label{Knopfmacher lemma}
Let $g(n)$ be a periodic arithmetical function with period $q>0$.  Then the sum $$\sum_{a=1}^q g(a)=0$$ if and only if  the following series
$$\sum_{n=1}^\infty  \frac{g(n) \log^k(n)}{n} $$ is convergent and  has the sum 
$$
\sum_{n=1}^\infty  \frac{g(n) \log^k(n)}{n} = \sum_{a=1}^{q} g(a) \gamma_k(a,  q),
$$
where $ \gamma_k (a,q)$ is defined as in \eqref{gamma_k(a,q)}.

\end{lemma}
To know more about the above generalized Euler constant $ \gamma_k(a,  q)$,  readers can see an interesting paper of Dilcher \cite{Dilcher}.

	In the coming section,  we present proof of our main results. 
	\section{Proof of Main Results}

		\begin{proof} [Proof of Theorem \ref{Laurent series coefficients for Riemann zeta function} ]
		From Lemma \ref{Integral representation of Riemann zeta function},  we know that
		\begin{equation} \label{Integral representation}
			\zeta(s) = \frac{s}{s-1} - s \int_{1}^{\infty} \frac{x- [x]}{x^{s+1}} dx , \hspace{0.4cm} \Re(s) > 0.  
		\end{equation}
		Now form the above equation it is clear that $\zeta(s)$ has a simple pole at $s=1$ with residue 1. Therefore,  it has a Laurent series expansion at $s=1$.  Let us suppose that 
		\begin{equation} \label{Laurent series of zeta}
			\zeta(s) = \frac{1}{s-1} + \sum_{k=0}^{\infty}{A_{k} (s-1)^k}
		\end{equation}
		Now our claim is that,  $$A_{k} = \frac{(-1)^k}{k!}{\gamma_{k}},  $$
		where $\gamma_k$ is defined as in \eqref{gamma_k}. 
   One can rewrite \eqref{Integral representation} as follows: 
		\begin{equation*} 
			\zeta(s) - \frac{1}{s-1}-1 = h(s),
			\end{equation*}
			where 
			\begin{equation*}
			h(s)= s \int_{1}^{\infty} \frac{[x]-x}{x^{s+1}} dx,
		\end{equation*}
		is analytic for $\Re(s)>0$.  Hence,  it will have a power series expansion around $s=1$ with $k$-th coefficient, $B_{k}= \frac{ h^k(1)}{k!}$,  that is,  
		\begin{equation} \label{Power series of h(s)}
			h(s)=  s \int_{1}^{\infty} \frac{[x]-x}{x^{s+1}} dx = \sum_{k=0}^{\infty}{B_{k} (s-1)^k}. 
		\end{equation}
		Now, by comparing the coefficients of \eqref{Laurent series of zeta} and  \eqref{Power series of h(s)}, we get $B_{0}=A_{0}-1$ and $B_{k}=A_{k}$, for all $k$. Therefore, our work is to finding the coefficients,$B_{k}$ i.e. the $k$-th derivative of $h(s)$ at $s=1$.  Let us simplify $h(s)$ as follows: 
{\allowdisplaybreaks		\begin{align*}
			h(s) &= s \int_{1}^{\infty} \frac{[x]-x}{x^{s+1}}dx\\
			&= s \lim_{r \to \infty} \int_{1}^{r} \frac{[x]-x}{x^{s+1}}dx\\
			&=  s \lim_{r \to \infty} \Bigg[\sum_{i=1}^{r-1} \int_{i}^{i+1} \frac{i}{x^{s+1}}dx - \int_{1}^{r} \frac{1}{x^s}dx \Bigg]\\
			&=  \lim_{r \to \infty} \Bigg[\sum_{i=1}^{r-1} \Big(\frac{i}{i^s} - \frac{i}{(i+1)^s}\Big) + \frac{s (r^{1-s}-1)}{s-1} \Bigg]\\
			&= \lim_{r \to \infty} \Bigg[\sum_{i=1}^{r-1} \frac{1}{i^s} - \frac{r-1}{r^s} + \frac{\Big((s-1)+1\Big) (r^{1-s}-1)}{s-1} \Bigg]\\
			&= \lim_{r \to \infty} \Bigg[\sum_{i=1}^{r-1} \frac{1}{i^s} + \frac{1}{r^s} - 1 + \frac{r^{1-s}}{s-1} - \frac{1}{s-1} \Bigg].  
		\end{align*}}
Further,  one can write 
		\begin{align*}
			- \frac{1}{s-1} + \frac{r^{1-s}}{s-1} &= - \frac{1}{s-1} + \frac{\exp(-(s-1) \log r)}{s-1}\\
			&= \sum_{l=1}^{\infty} \frac{(-1)^l (s-1)^{l-1}}{l !} \log^l(r).
		\end{align*}
		Utilizing this expression,  we obtain
$$h(s) = \lim_{r \to \infty} \Bigg[ \sum_{i=1}^{r} \frac{1}{i^s} - 1 + \sum_{l=1}^{\infty} (-1)^l \frac{ (s-1)^{l-1} \log^l r}{l !}  \Bigg].$$ 
		Now differentiating $k$-times at $s=1$,  one can see that 
		\begin{align*}
			h^{(k)} (1) &= \lim_{r \to \infty} \Bigg[  (-1)^k \sum_{i=1}^{r} \frac{\log^k i}{i} + (-1)^{k+1} \frac{k! \log^{k+1}r}{(k+1)!}  \Bigg] \\
			&= (-1)^k \lim_{r \to \infty} \Bigg[\sum_{i=1}^{r} \frac{\log^k i}{i} - \frac{\log^{k+1}r}{k+1} \Bigg].
		\end{align*}
		Therefore,  from the definition \eqref{gamma_k} of $\gamma_k$,  it follows that $B_{k} k! = (-1)^k \gamma_{k}.$ This completes the proof of Theorem \ref{Laurent series coefficients for Riemann zeta function}. 
	\end{proof}

	\begin{proof} [Proof of Theorem \ref{Laurent series coefficients for Hurwitz zeta function}]
From Lemma \ref{Integral representation of Hurwitz zeta function},  it is evident that $\zeta(s,a)$ has a simple pole at $s=1$ with residue $1$.  Let us suppose we write
 \begin{equation}\label{Laurent series of zeta(s,a)}
			\zeta(s,a)= \frac{1}{s-1} + \sum_{k=0}^{\infty} B_{k} (s-1)^{k},
		\end{equation}
		then our claim is that  $$B_{k}= \frac{(-1)^{k}}{k!} \gamma_{k}(a),$$
		where 
\begin{equation}\label{gamma_k(a)}
		\gamma_{k}(a)= \lim_{p \to \infty} \left\{ \sum_{i=0}^{p} \frac{(\log(i+a))^{k}}{(i+a)} - \frac{(\log(p+a))^{k+1}}{k+1} \right\}. 
		\end{equation}
		Let us consider,   for $N \in \mathbb{N}$,  $0<a \leq 1$,  $\Re(s)>0$,  an analytic function
		\begin{equation*}
			H_N(s,  a) = s \int_{N}^{\infty}\frac{[x]-x}{(x+a)^{s+1}} dx.  
		\end{equation*}
		Then we can write
{\allowdisplaybreaks	
		\begin{align*}
			H_N(s,  a)
			&= s \int_{N}^{\infty}\frac{[x]-x}{(x+a)^{s+1}}  dx\\
			&= s \lim_{p \to \infty} \left[ \int_{N}^{p}\frac{[x]}{(x+a)^{s+1}}dx - \int_{N}^{p}\frac{ x}{(x+a)^{s+1}}dx \right]\\
			&= s \lim_{p \to \infty}\left[ \sum_{i=N}^{p-1} \int_{i}^{i+1}\frac{i}{(x+a)^{s+1}}dx -\int_{N}^{p}\frac{ (x+a)-a}{(x+a)^{s+1}}dx \right]\\
			&=s \lim_{p \to \infty} \left[\sum_{i=N}^{p-1} \int_{i}^{i+1}\frac{i}{(x+a)^{s+1}}dx -\int_{N}^{p}\frac{1}{(x+a)^{s}}dx + a\int_{N}^{p}\frac{1}{(x+a)^{s+1}}dx \right]\\
			&=  \lim_{p \to \infty} \Bigg[ \sum_{i=N}^{p-1} \left(  \frac{i}{(i+a)^s} - \frac{ i}{(i+1+a)^s} \right) + \frac{s}{s-1} \left( \frac{1}{(p+a)^{s-1}} - \frac{1}{(N+a)^{s-1}}\right)  \\
			& \hspace{6.4cm}- a \left(\frac{1}{(p+a)^s}- \frac{1}{(N+a)^{s}}\right)\Bigg] \\
			&:=  \lim_{p \to \infty}  I_1(p) + I_2(p),
		\end{align*}}
		where
		\begin{align*}
			I_{1}(p)
			&=  \sum_{i=N}^{p-1} \left( \frac{i}{(i+a)^s}-\frac{i}{(i+1+a)^s} \right)
			=   \sum_{i=N}^{p-1} \frac{1}{(i+1+a)^s}+ \frac{N}{(N+a)^s} - \frac{p}{(p+a)^s}, \\
			I_{2}(p) 
			&= \frac{s}{s-1} \left( \frac{1}{(p+a)^{s-1}} - \frac{1}{(N+a)^{s-1}}\right) - a \left(\frac{1}{(p+a)^s}- \frac{1}{(N+a)^{s}}\right)\\
			&= \frac{(s-1)+1}{s-1} \left( \frac{1}{(p+a)^{s-1}} - \frac{1}{(N+a)^{s-1}}\right) - a \left(\frac{1}{(p+a)^s}- \frac{1}{(N+a)^{s}}\right).
		\end{align*}		 
	Now substituting the above expressions of $I_1(p)$ and $I_2(p)$ and after simplifying a bit, we get the following expression of $H_N(s,a)$, 
		\begin{align*}
			H_N(s,a) &= \lim_{p \to \infty} \left[\sum_{i=N}^{p-1} \frac{1}{(i+1+a)^s} + \frac{1}{s-1} \left( \frac{1}{(p+a)^{s-1}} - \frac{1}{(N+a)^{s-1}}\right)\right]
		\end{align*}
		Using the above expression of $H_N(s,a)$ in Lemma \ref{Integral representation of Hurwitz zeta function}, we arrive
			$$\zeta(s,a)= \sum_{n=0}^{N}\frac{ 1}{(n+a)^s} + \frac{(N+a)^{1-s}}{s-1} + \lim_{p \to \infty}\left[ \sum_{i=N}^{p-1} \frac{1}{(i+1+a)^s} + \frac{1}{s-1} \left(\frac{1}{(p+a)^{s-1}} - \frac{1}{(N+a)^{s-1}}\right)\right].$$
Upon simplification,  one can see that
			$$\zeta(s,a) = \lim_{p \to \infty} \left[\sum_{i=0}^{p} \frac{1}{(i+a)^s} + \frac{1}{(s-1)} \frac{1}{(p+a)^{s-1}}\right].$$
Now we write 
		\begin{align*}
			\frac{1}{(s-1)} (p+a)^{1-s} 
			&= \exp(-(s-1) \log(p+a)) \\
			&= \sum_{m=0}^{\infty} \frac{(-1)^m}{m!} (s-1)^{m-1} (\log(p+a))^{m}\\
			&= \frac{1}{s-1} + \sum_{m=1}^{\infty} \frac{(-1)^m}{m!} (s-1)^{m-1} (\log(p+a))^{m}.
		\end{align*}
		Therefore, 
	\begin{align}
			\zeta(s,a) &= \lim_{p \to \infty} \left[ \sum_{i=0}^{p} \frac{1}{(i+a)^s}   + \frac{1}{s-1} + \sum_{m=1}^{\infty} \frac{(-1)^m}{m!} (s-1)^{m-1} (\log(p+a))^{m} \right] \nonumber \\
		\Rightarrow 	\zeta(s,a) - \frac{1}{s-1} &= \lim_{p \to \infty} \left[\sum_{i=0}^{p} \frac{1}{(i+a)^s} + \sum_{m=1}^{\infty} \frac{(-1)^m}{m!} (s-1)^{m-1} (\log(p+a))^{m} \right] \nonumber \\
		& := F(s).  \label{F(s)}
	\end{align}
		Note that the function $F(s)$ is analytic at $s=1$,  so it will have a power series expansion around $s=1$.  From \eqref{Laurent series of zeta(s,a)},  we have 
		\begin{equation*}
			F(s) = \sum_{k=0}^{\infty} B_{k} (s-1)^{k}, 
		\end{equation*} 	  
		where the $k$-th coefficient $B_{k}$ is $ \frac{F^{(k)}(1)}{k!}$.  
		Now using the definition \eqref{F(s)} of $F(s)$ and differentiating $k$-times at $s=1$,  we see that 
		\begin{align*}
			F^{k} (1) &= \lim_{p \to \infty} \left[ \sum_{i=0}^{p} \frac{ (-1)^k (\log(i+a))^{k}}{(i+a)} + \frac{(-1)^{k+1}}{(k+1)}\ (\log(p+a))^{k+1} \right]\\
			&= (-1)^{k}   \lim_{p \to \infty} \left[ \sum_{i=0}^{p} \frac{(\log(i+a))^{k}}{(i+a)} - \frac{(\log(p+a))^{k+1}}{k+1} \right].  
		\end{align*}
		Finally,  using the definition \eqref{gamma_k(a)} of $\gamma_k(a)$,  it is clear that $k! B_k= (-1)^k \gamma_k(a)$.  
		Thus,  it completes the proof of our claim.	 
	\end{proof}
	
	\begin{proof} [Proof of Theorem \ref{Laurent series coefficients for Dirichlet l=function}]
It is well known that $L(s,  \chi)$ can be analytically continued to the whole complex plane as we are dealing with non-principal Dirichlet character $\chi$.  Thus,  we will have the following Laurent series expansion at $s=1$:
\begin{align}\label{Laurent series of L(s,chi)}
L(s,  \chi)= \sum_{k=0}^\infty L^{(k)}(1, \chi) \frac{(s-1)^k}{k!}.
\end{align}
Let us define $A(x)= \sum_{n \leq x} \chi(n)$.   Using Abel's summation formula,  for $\Re(s)>0$,  one can show that  $L(s,\chi)$ satisfies the following integral representation: 
	\begin{align*}
		L(s,\chi) &= s \int_{1}^{\infty} \frac{A(x)}{x^{s+1} } dx   \\
		&= ((s-1)+1) \int_{1}^{\infty} \frac{A(x)}{x^2} x^{-s+1} dx\\
		&= ((s-1)+1) \int_{1}^{\infty} \frac{A(x)}{x^2} \sum_{m=0}^{\infty}\frac{(-1)^m (\log x)^m}{m!} (s-1)^m dx \\
		&= \int_{1}^{\infty} \frac{A(x)}{x^2} \sum_{m=0}^{\infty}\frac{(-1)^m (\log x)^m}{m!} (s-1)^{m+1} dx  \\
		& + \int_{1}^{\infty} \frac{A(x)}{x^2} \sum_{m=0}^{\infty}\frac{(-1)^m (\log x)^m}{m!} (s-1)^m dx. 
	\end{align*}
From the above expression,  one can check that
	\begin{align}
		L^{(k)} (1,  \chi) &= \int_{1}^{\infty} \frac{A(x)}{x^2} \frac{(-1)^{k-1}}{(k-1)!} k! (\log x)^{k-1} dx + \int_{1}^{\infty} \frac{A(x)}{x^2} \frac{(-1)^{k}}{k!} k! (\log x)^{k} dx \nonumber \\
		&= (-1)^k \Bigg[ \int_{1}^{\infty} \frac{A(x)}{x^2} (\log x)^{k} dx - k \int_{1}^{\infty} \frac{A(x)}{x^2} (\log x)^{k-1} dx \Bigg]  \nonumber \\
		&= (-1)^k \lim_{r \to \infty}\Bigg[ \int_{1}^{r} \frac{A(x)}{x^2} \Big( (\log x)^{k} - k (\log x)^{k-1} \Big) dx \Bigg].  \label{kth derivative of L(s,chi) at 1}
	\end{align} 
Using Abel's summation formula,   we obtain
	\begin{align}
		\sum_{1<n \leq x}\chi(n) \frac{\log^{k}(n)}{n} &= A(x)  \frac{\log^{k}(x)}{x} - A(1) f(1) - \int_{1}^{x} A(t) f'(t) dt \nonumber \\
		&= A(x) \frac{\log^{k}x}{x} + \int_{1}^{x} A(t) \Big(\frac{\log^{k} t - k \log^{k-1} t}{t^2} \Big) dt.   \nonumber
\end{align}
Since $|A(x)|< \phi(q)$,  letting $x \rightarrow \infty$ in the above expression,  we see that
\begin{align}		
 \lim_{x \rightarrow \infty }\sum_{1<n \leq x}\chi(n) \frac{\log^{k}(n)}{n}		&= \lim_{x\rightarrow \infty} \int_{1}^{x} A(t) \Big(\frac{\log^{k} t - k \log^{k-1} t}{t^2} \Big) dt.  \label{Use of Able identity}
	\end{align}
Thus,  from \eqref{kth derivative of L(s,chi) at 1} and \eqref{Use of Able identity},  it is evident that
	\begin{align}
		L^{(k)} (1,  \chi) &= (-1)^k \lim_{x \to \infty} \sum_{n\leq x}\chi(n) \frac{\log^{k}(n)}{n} \nonumber \\
		&= (-1)^k \sum_{n=1}^{\infty} \chi(n) \frac{\log^{k}(n)}{n}  \nonumber \\ 
		&= (-1)^k \sum_{a=1}^{q} \chi(a) \gamma_k (a,q),   \label{kth derivative}
	\end{align} 
here in the final step we used Lemma \ref{Knopfmacher lemma} since $\chi$ is a non-principal character of modulus $q$ and $\sum_{a=1}^{q} \chi(a)=0$ holds.  Finally,  substituting \eqref{kth derivative} in \eqref{Laurent series of L(s,chi)},   we complete the proof of  Theorem \ref{Laurent series coefficients for Dirichlet l=function}.
	\end{proof}

\begin{proof} [Proof of Theorem \ref{Laurent series coefficients for $L$-Function associated to cusp form of weight $k$}]
	Using the definition of the gamma function and $L(f,s)$,  for $\Re(s)> \frac{k+1}{2}$,  one can show that 
\begin{align}\label{1st step}
(2 \pi)^{-s} \Gamma(s) L(f,  s) = \int_{0}^\infty W(y) y^{s-1} dy,
\end{align}
	where 
	\begin{align}\label{W(y)}
W(y)= \sum_{n=1}^{\infty} a_f (n) \exp(-2\pi ny).  
\end{align}
Now we divide the integral \eqref{1st step} in two parts, mainly,  we write as follows:
\begin{align}
(2 \pi)^{-s} \Gamma(s) L(f,  s) &  = \int_{0}^1 W(y) y^{s-1} dy + \int_{1}^{\infty} W(y) y^{s-1} dy \nonumber \\
& := I_1(y,  s) + I_2(y,  s). \label{I1 + I2}
\end{align}
Note that $I_2(y,  s)$ is an entire function in the variable $s$ as it can be shown that $W(y)=O\left(e^{-2\pi y} \right)$ as $y \rightarrow \infty$.  Now to simplify $I_1(y,  s)$, we use transformation formula   \eqref{Analogue of Jacobi theta transformation} for $W(y)$.  Thus,  using  \eqref{Analogue of Jacobi theta transformation},  we see that
\begin{align*}
I_1(y,  s) =  \int_{0}^1 W(1/y) y^{s-1-k} dy.  
\end{align*}
Changing the variable $y \rightarrow 1/y$ in the above integration,  it yields that
\begin{align}\label{Final_I1(y)}
I_1(y,  s)= \int_{1}^\infty W(y) y^{k-s-1} dy. 
\end{align}
Now one can clearly show that $I_1(y,  s)$ is entire in $s$.
Substituting \eqref{Final_I1(y)} in \eqref{I1 + I2},  we obtain
\begin{align}\label{analytic continuation}
(2 \pi)^{-s} \Gamma(s) L(f,  s) = \int_{1}^\infty  W(y) \left(y^{k-1-s}+ y^{s-1}   \right) dy.  
\end{align}
The above right side integral is an entire function in the variable $s$.  Thus,  it gives an analytic continuation of $L(f,s)$ to the whole complex plane.  
	Now one can write 
	\begin{align}\label{power series}
	 y^{k-1-s}+ y^{s-1} =  \frac{1}{y} \sum_{n=0}^\infty \left\{ y^k (-1)^n \log^n(y) + \log^n(y) \right\} \frac{s^n}{n!}. 
	\end{align}
This power series converges absolutely for any complex $s$. Therefore,  substituting \eqref{power series} in \eqref{analytic continuation} and interchanging summation and integration,  we derive that 
\begin{align}\label{1st power series}
(2 \pi)^{-s} \Gamma(s) L(f,  s) = \sum_{n=0}^\infty A(n,  k) \frac{s^n}{n!},  
\end{align}
	where 
	\begin{align}\label{A(n,k)}
	A(n,  k) = \int_{1}^\infty \frac{W(y)}{y} \left\{ y^k (-1)^n \log^n(y) + \log^n(y) \right\} dy.  
	\end{align}
Now let us suppose that 
\begin{align}\label{2nd power series}
\frac{(2\pi)^s}{\Gamma(s)} =  \sum_{n=1}^\infty B(n) \frac{s^n}{n!}.
\end{align} 	
To calculate $B(n)$,  we can use the fact that $$1/\Gamma(s)= s+ \gamma s^2 + \left( \frac{\gamma^2}{2}- \frac{\pi^2}{12} \right)s^3 + \cdots .$$
One can easily check that $B(1)=1,  B(2)=2( \gamma + \log(2\pi))$.  Finally,  substituting \eqref{2nd power series} in \eqref{1st power series},  we obtain
\begin{align*}
L(f,  s) = \sum_{n=1}^\infty C(n,  k) s^n,  
\end{align*}
where 
\begin{align}\label{defn_C(n,k)}
C(n, k) = \sum_{ \substack{0 \leq i ,  1\leq j, \\ i +j = n} }  \frac{A(i,  k)}{i!} \frac{ B(j)}{j!}.  
\end{align}
First,  let us try to find the coefficient of $s$ in the above Laurent series expansion of $L(f,s)$ at $s=0$.  Note that $C(1, k)= A(0,  k) B(1)$.   Thus,  $C(1,  k) = A(0, k)$ as $B(1)=1$.  
Now we shall try to simplify $A(0,k)$.  From \eqref{A(n,k)},  we have
\begin{align*}
A(0,  k) = \int_{1}^\infty \frac{W(y)}{y} \left\{ y^k   + 1 \right\} dy. 
\end{align*}
From the definition \eqref{W(y)} of $W(y)$,  one can easily check that the series is absolutely and uniformly convergent in any compact subset of $\Re(y)>0$.  Thus,  using \eqref{W(y)}  and interchanging summation and integration,  we arrive 
\begin{align*}
A(0,  k) = \sum_{n=1}^\infty a_f(n) \left( \int_{1}^\infty e^{-2\pi n y}  y^{k}  \frac{ dy}{y}  + \int_{1}^\infty  e^{-2\pi n y}  \frac{ dy}{y}    \right). 
\end{align*}
	Now we change the variable $y$ by $y/2n \pi$ to see that
	\begin{align}
	A(0,  k)  & = \sum_{n=1}^\infty a_f(n) \left( \int_{2n \pi}^\infty e^{-y} y^{k-1}  \frac{dy}{(2n\pi)^k} + \int_{2n \pi}^\infty e^{-y} y^{-1}  dy  \right) \nonumber \\
	& = \sum_{n=1}^\infty \frac{ a_f(n)}{(2n\pi)^k} \left(  \Gamma(k,  2 n \pi)  +  (2n\pi)^k\Gamma(0,  2 n \pi) \right) \nonumber \\
	& =   \sum_{n=1}^\infty \frac{ a_f(n) \Gamma(k,  2 n \pi) }{(2n\pi)^k} + \sum_{n=1}^\infty  a_f(n) \Gamma(0,  2 n \pi). \label{final A(0,k)}
	\end{align}
	In the penultimate step, we  used the definition of the incomplete gamma function \eqref{incomplete gamma}.  Using the asymptotic \eqref{asymp1} of the incomplete gamma function,  one can check that both of the above infinite series are convergent.  Now we shall try to evaluate $C(2,k)$.  From \eqref{defn_C(n,k)},  it is clear that 
	\begin{align}
	C(2,k) & = A(0,k) \frac{ B(2)}{2!} + A(1,k) B(1) \nonumber \\
	&= A(0,k) ( \gamma + \log(2\pi) ) + A(1,k).  \label{C(2,k)}
	\end{align}
We have already calculated $A(0,k)$ in \eqref{final A(0,k)}.  Thus,  we need to calculate only $A(1,k)$.   From the definition \eqref{A(n,k)} of $A(n,k)$,  we know that
\begin{align*}
A(1,  k) = \int_{1}^\infty \frac{W(y)}{y} \left\{ - y^k  \log(y) + \log(y) \right\} dy.  
\end{align*}
Employing the definition \eqref{W(y)} of $W(y)$,  one can see that 
\begin{align*}
A(1,  k) &=  \sum_{n=1}^\infty a_f(n) \left( - \int_{1}^\infty e^{-2\pi n y}  y^{k} \log(y)  \frac{ dy}{y}  
+ \int_{1}^\infty  e^{-2\pi n y} \log(y)  \frac{ dy}{y}    \right). 
\end{align*}
Again,  changing the variable variable $y$ by $y/2n \pi$,  it yields that
\begin{align}\label{final A(1,k)}
A(1,  k) =   \sum_{n=1}^\infty a_f(n)&  \Bigg( \int_{2n \pi}^\infty e^{-y} y^{k-1} \{ - \log(y)+\log(2n\pi) \}  \frac{dy}{(2n\pi)^k} \nonumber  \\
& + \int_{2n \pi}^\infty e^{-y} \{ \log(y)-\log(2n\pi) \}   dy  \Bigg). 
\end{align}
Now using the definition \eqref{gen incomplete gamma} of the incomplete gamma function,  we see that 
\begin{align}
A(1,k) & = \sum_{n=1}^\infty a_f(n) \left(  \frac{-\Gamma_1(k,  2n \pi) + \log(2n \pi) \Gamma(k,  2n\pi)}{(2n\pi)^k}   \right) \nonumber \\
& +  \sum_{n=1}^\infty a_f(n) \left( \Gamma_1(0,  2n \pi) -\log(2n \pi) \Gamma(0,  2n\pi)  \right).  \label{A(1,k)}
\end{align}
Using the asymptotic expansions \eqref{asymp1}, \eqref{asymp2} of $\Gamma(s,  a),  \Gamma_1(s,a)$,  one can easily show that the series present in \eqref{final A(1,k)} are absolutely convergent.  Finally,  substituting \eqref{A(1,k)} in \eqref{C(2,k)},  we get the final expression for $C(2,k)$. 
In a similar fashion,  one can calculate $C(n, k)$ for $n \geq 3$. 
 This finishes the proof of  Theorem \ref{Laurent series coefficients for $L$-Function associated to cusp form of weight $k$}.  
\end{proof}

\section{Numerical verification of Theorem \ref{Laurent series coefficients for $L$-Function associated to cusp form of weight $k$}} Using Mathematica software,  we numerically verified our Theorem \ref{Laurent series coefficients for $L$-Function associated to cusp form of weight $k$} for Ramanujan cusp form $\Delta(z)$,  which is an example of a weight $12$ cusp form over $SL_2(\mathbb{Z})$.  To evaluate $C(n, 12)$,  we used only $30$ terms of each infinite series present in \eqref{Final C(1,k)}-\eqref{Final C(2,k)}. 

\begin{center}

\begin{tabular}{ |p{1cm}| p{5 cm}|p{5 cm}|  }
	\hline
	\multicolumn{3}{|c|}{Special values} \\
	\hline
	$n$ & $\frac{L^{(n)}(\Delta,  0)}{n!}$ & $C(n,  12)$\\
	\hline
	1& 0.01048627312924115  & 0.01048627312924116  \\
	\hline
	2&  0.01894504907238154   & 0.01894525791618929\\
	\hline
\end{tabular}
\end{center}
	
\section{Acknowledgement}
	The last author wishes thank Science and Engineering Research Board (SERB),  India,   for giving MATRICS grant (File No. MTR/2022/000545) and SERB CRG grant (File No.  CRG/2023/002122).

\end{document}